\documentclass[10pt] {bcp}

\usepackage{amsfonts}

\newtheorem{theorem}{Theorem}[section]

\newtheorem{corollary}[theorem]{Corollary}

\theoremstyle{definition}

\usepackage{amsmath,amsthm,amssymb,amscd}
\usepackage{amsthm}

\begin{document}

\mathclass{Primary 41A36.}

\abbrevauthors{A.S. Serdyuk, T.A. Stepanyuk} 
\abbrevtitle{Uniform approximations  \\
by  Fourier sums}

\title{
 Uniform approximations  \\
by  Fourier sums on  classes \\  of  convolutions of periodic functions
 $ \ $}

%%%% Note the way of obtaining lower case symbols in the title %%%

\author{Anatoly S. Serdyuk}
\address{Institute of Mathematics NAS of
Ukraine \\
 Tereschenkivska st. 3,  01601   Kyiv, Ukraine\\
E-mail: serdyuk@imath.kiev.ua}

\author{Tetiana A. Stepanyuk}
\address{
Johann Radon Institute for Computational   and Applied Mathematics (RICAM)
\\ Austrian Academy of Sciences, 
Altenbergerstrasse 69 4040, Linz, Austria;\\
   Institute of Mathematics of NAS of Ukraine, \\ 3, Tereshchenkivska st., 01601, Kyiv-4, Ukraine \\
E-mail: tania$_{-}$stepaniuk@ukr.net}

\maketitlebcp

{\bf Summary}

\noindent 
We establish asymptotic estimates for exact upper bounds of uniform approximations by Fourier sums on the classes of $2\pi$--periodic functions, which are represented by convolutions of functions $\varphi (\varphi\bot 1)$ from unit ball of the space $L_{1}$ with fixed kernels $\Psi_{\beta}$ of the form
$\Psi_{\beta}(t)=\sum\limits_{k=1}^{\infty}\psi(k) \cos\left(kt-\frac{\beta\pi}{2}\right)$, $\sum\limits_{k=1}^{\infty}k\psi(k)<\infty$, $\psi(k)\geq 0$, $\beta\in\mathbb{R}$.

%of generalized Poisson integrals  $C^{\alpha,r}_{\beta}L_{p}$, $1\leq p <\infty$, we obtain upper estimates for the deviations of Fourier sums in the uniform metric in terms of the best approximations of the  generalized derivatives $f^{\alpha,r}_{\beta}$ of functions of this kind by trigonometric polynomials in the metric  of the spaces $L_{p}$. Obtained estimates are asymptotically best possible.

\vskip 5mm

\section{Introduction}

Let $L_{1}$
 be the space of $2\pi$--periodic functions $f$ summable 
on  $[0,2\pi)$, in which
the norm is given by the formula
$\|f\|_{1}=\int\limits_{0}^{2\pi}|f(t)|dt$; $L_{\infty}$ be the space of measurable and essentially bounded   $2\pi$--periodic functions  $f$ with the norm
$\|f\|_{\infty}=\mathop{\rm{ess}\sup}\limits_{t}|f(t)|$; $C$ be the space of continuous $2\pi$--periodic functions  $f$, in which the norm is specified by the equality
 ${\|f\|_{C}=\max\limits_{t}|f(t)|}$.

%Assume that $f$ is a $2\pi$--periodic function, summable on $[0,2\pi)$ with the Fourier series
%\begin{equation}\label{FourierSer}
%S[f]=\frac{a_{0}}{2}+\sum\limits_{k=1}^{\infty}\left( a_{k}\cos kx +b_{k}\sin kx \right)
%=\sum\limits_{k=0}^{\infty}A_{k}(f,x)
 %\end{equation}
 
Let $\psi(k)$ be an arbitrary fixed sequence of real, nonnegative numbers and let $\beta$ be a fixed real number. 

We set
\begin{equation}
B_{1}^{0}:=\left\{\varphi: \ ||\varphi||_{1}\leq 1, \  \varphi\perp1\right\}.
\end{equation}

Further let $C^{\psi}_{\beta,1}$ be the set of all functions $f$, which are represented for all $x$ as convolutions of the form 
\begin{equation}\label{conv}
f(x)=\frac{a_{0}}{2}+\frac{1}{\pi}\int\limits_{-\pi}^{\pi} \varphi(t)\Psi_{\beta}(x-t)dt,
\ a_{0}\in\mathbb{R}, \ \varphi\in B_{1}^{0}, \
\end{equation}
where $\Psi_{\beta}$ is a fixed kernel of the form
\begin{equation}\label{kernelPsi}
\Psi_{\beta}(t)=\sum\limits_{k=1}^{\infty}\psi(k)\cos
\big(kt-\frac{\beta\pi}{2}\big), \ \psi(k)\geq 0, \  \beta\in
    \mathbb{R},
\end{equation}
 and the following condition holds:
 \begin{equation}\label{condition}
\sum\limits_{k=1}^{\infty}\psi(k)<\infty.
\end{equation}

Condition (\ref{condition}) provides an embedding $C^{\psi}_{\beta,1}\subset C$.

For $\psi(k)=e^{-\alpha k^{r}}$, $\alpha, r>0$, the kernels $\Psi_{\beta}(t)$ of the form (\ref{kernelPsi}) are called generalized Poisson kernels
 $P_{\alpha,r,\beta}(t)=\sum\limits_{k=1}^{\infty}e^{-\alpha k^{r}}\cos
\big(kt-\frac{\beta\pi}{2}\big)$ and the classes of functions $f$, generated by these kernels are the classes of generalized Poisson integrals $C^{\alpha,r}_{\beta,1}$.

For the classes $C^{\psi}_{\beta,1}$ we consider the quantities
   \begin{equation}\label{FourierSum}
 {\cal E}_{n}(C^{\psi}_{\beta,1})_{C}=\sup\limits_{f\in
C^{\psi}_{\beta,1}}\|f(\cdot)-S_{n-1}(f;\cdot)\|_{C},   
  \end{equation}
  where $S_{n-1}(f;\cdot)$ are the partial Fourier sums of order  $n-1$ for a function $f$.
 
 Approximations by Fourier sums on other classes of differentiable functions in uniform metric were investigated in works \cite{Kol}--\cite{Teljakovsky1968}.
 
In this paper we consider Kolmogorov--Nikolsky problem about finding of asymptotic equalities of the quantity (\ref{FourierSum}) as $n\rightarrow\infty$.  
 
\section{Main result}

The following statement holds.

\begin{theorem}\label{theorem1}
Let $\sum\limits_{k=1}^{\infty}k\psi(k)<\infty$, $\psi(k)\geq 0$, $k=1,2,...$ and $\beta\in\mathbb{R}$. Then as $n\rightarrow\infty$ the following asymptotic equality holds
\begin{equation}\label{Theorem1Asymp}
{\cal E}_{n}(C^{\psi}_{\beta,1})_{C}=
\frac{1}{\pi}
\sum\limits_{k=n}^{\infty}\psi(k)+
 \frac{\mathcal{O}(1)}{n} \sum\limits_{k=1}^{\infty}k\psi(k+n),
 \end{equation}
where  $O(1)$ is a quantity uniformly bounded in all parameters.
\end{theorem}

%\emph{Proof of Theorem 1.}

\begin{proof} 
According to  (\ref{conv}) and (\ref{FourierSum}) we have that
\begin{equation}\label{f1}
{\cal E}_{n}(C^{\psi}_{\beta,1})_{C}=
\frac{1}{\pi}\sup\limits_{\varphi\in B_{1}^{0}}\bigg\|\int\limits_{-\pi}^{\pi} \varphi(t)\Psi_{\beta,n}(x-t)dt\bigg\|_{C}, 
\end{equation}
 where
 \begin{equation}\label{f2}
 \Psi_{\beta,n}(t):=
\sum\limits_{k=n}^{\infty}\psi(k)\cos\Big(kt-\frac{\beta\pi}{2}\Big),  \  \ \beta\in\mathbb{R}.
\end{equation}
Taking into account the invariance of the sets  $B_{1}^{0}$ under shifts of the argument,
 from (\ref{f1}) we conclude that
\begin{equation}\label{f3}
{\cal E}_{n}(C^{\psi}_{\beta,1})_{C}=
\frac{1}{\pi}\sup\limits_{\varphi\in B_{1}^{0}}\int\limits_{-\pi}^{\pi}  \varphi(t) \Psi_{\beta,n}(t)dt.
\end{equation}

 On the basis of  the duality relation  (see, e.g., \cite[Chapter 1, Section 1.4]{Korn}) we obtain
\begin{equation}\label{f4}
\sup\limits_{\varphi\in B_{1}^{0}}\int\limits_{-\pi}^{\pi}\Psi_{\beta,n}(t)\varphi(t)dt=
\inf\limits_{\lambda\in\mathbb{R}}\|\Psi_{\beta,n}(t)-\lambda\|_{C}, \
\end{equation}

We represent the function $ \Psi_{\beta,n}(t)$, which is defined by formula  (\ref{f2}), in the form
\begin{equation}\label{pp}
  \Psi_{\beta,n}(t)=
g_{\psi,n}(t)\cos\Big(nt-\frac{\beta\pi}{2}\Big)+h_{\psi,n}(t)\sin\Big(nt-\frac{\beta\pi}{2}\Big),
\end{equation}
where
\begin{equation}\label{g}
  g_{\psi,n}(t):=
\sum\limits_{k=0}^{\infty}\psi(k+n)\cos kt,
\end{equation}
\begin{equation}\label{h}
 h_{\psi,n}(t):=
-\sum\limits_{k=0}^{\infty}\psi(k+n)\sin kt.
\end{equation}

It is obvious that
\begin{equation}\label{f5}
\inf\limits_{\lambda\in\mathbb{R}}\|\Psi_{\beta,n}(t)-\lambda\|_{C}
\leq
\|\Psi_{\beta,n}\|_{C}
\end{equation}
and 
\begin{equation}\label{f6}
\frac{1}{2}\left\|\Psi_{\beta,n}\left(t+\frac{\pi}{n}\right)-\Psi_{\beta,n}(t)\right\|_{C}
\leq
\inf\limits_{\lambda\in\mathbb{R}}\|\Psi_{\beta, n}(t)-\lambda\|_{C}.
\end{equation}

In view of representation  (\ref{pp}) and applying mean value theorem, we obtain that
\begin{align}\label{f7}
&\frac{1}{2}\left\|\Psi_{\beta,n}\left(t+\frac{\pi}{n}\right)-\Psi_{\beta,n}(t)\right\|_{C}
\notag \\
=&
\frac{1}{2}\Big\|2\Psi_{\beta,n}(t)+
g_{\psi,n}\Big(t+\frac{\pi}{n}\Big)\cos\Big(nt-\frac{\beta\pi}{2}\Big)
+h_{\psi,n}\Big(t+\frac{\pi}{n}\Big)\sin\Big(nt-\frac{\beta\pi}{2}\Big)\notag \\
-& \left(g_{\psi,n}(t)\cos\Big(nt-\frac{\beta\pi}{2}\Big)
+h_{\psi,n}(t)\sin\Big(nt-\frac{\beta\pi}{2}\Big)\right)\Big\|_{C} \notag \\
=&
\|\Psi_{\beta,n}\|_{C} +
\mathcal{O}(1) \left(\left \| g_{\psi,n}\Big(t+\frac{\pi}{n}\Big)- g_{\psi,n}(t) \right \|_{C} + 
\left \| h_{\psi,n}\Big(t+\frac{\pi}{n}\Big)- h_{\psi,n}(t) \right \|_{C} \right) \notag \\
=&
\|\Psi_{\beta,n}\|_{C} +
\mathcal{O}(1)\left( \frac{1}{n} \left \| g_{\psi,n}' \right \|_{C} + 
\frac{1}{n}\left \| h_{\psi,n}' \right \|_{C} \right) 
\notag \\
=&
\|\Psi_{\beta,n}\|_{C} +
 \frac{\mathcal{O}(1)}{n} \sum\limits_{k=1}^{\infty}k\psi(k+n).
\end{align}

So,  formulas  (\ref{f3}),  (\ref{f4}) and  (\ref{f5})-- (\ref{f7})  imply
\begin{equation}\label{f8}
{\cal E}_{n}(C^{\psi}_{\beta,1})_{C}=
\frac{1}{\pi}
\|\Psi_{\beta,n}\|_{C} +
 \frac{\mathcal{O}(1)}{n} \sum\limits_{k=1}^{\infty}k\psi(k+n).
\end{equation}

The kernel $\Psi_{\beta,n}$ can be written in the form
\begin{align}\label{f8}
 \Psi_{\beta,n}(t)
 & =
  \sqrt{g_{\psi,n}^{2}(t)+h_{\psi,n}^{2}(t)}  \times  \notag \\ 
  &\times
  \left(\frac{g_{\psi,n}(t)}{ \sqrt{g_{\psi,n}^{2}(t)+h_{\psi,n}^{2}(t)}} \cos\Big(nt-\frac{\beta\pi}{2}\Big)+
  \frac{h_{\psi,n}(t)}{ \sqrt{g_{\psi,n}^{2}(t)+h_{\psi,n}^{2}(t)}} \sin\Big(nt-\frac{\beta\pi}{2}\Big)
  \right)
  \notag \\
 &=
   \sqrt{g_{\psi,n}^{2}(t)+h_{\psi,n}^{2}(t)} 
 \cos\left(nt-\frac{\beta\pi}{2}-arg(g_{\psi,n}(t)+i h_{\psi,n}(t))\right).
\end{align}

Let
\begin{equation}\label{f8}
t_{0}:=\frac{1}{n}\left( \frac{\beta\pi}{2}+arg(g_{\psi,n}(t)+i h_{\psi,n}(t)) \right).
\end{equation}
Then
\begin{equation}\label{f9}
\|\Psi_{\beta,n}\|_{C}\geq
\Psi_{\beta,n}(t_{0})
=  \sqrt{g_{\psi,n}^{2}(t_{0})+h_{\psi,n}^{2}(t_{0})} \geq
|g_{\psi,n}(t_{0}) |.
\end{equation}

Using mean value theorem we have that
\begin{align}\label{f10}
&|g_{\psi,n}(t_{0}) |= g_{\psi,n}(0)+ |g_{\psi,n}(t_{0})-  g_{\psi,n}(0)|=
g_{\psi,n}(0)+ \frac{\mathcal{O}(1)}{n}  \| g_{\psi,n}' \|_{C} \notag \\
=& \sum\limits_{k=n}^{\infty}\psi(k)+
 \frac{\mathcal{O}(1)}{n} \sum\limits_{k=1}^{\infty}k\psi(k+n).
\end{align}

On the other hand it is clear that
\begin{equation}\label{f11}
\|\Psi_{\beta,n}\|_{C}\leq
\sum\limits_{k=n}^{\infty}\psi(k).
\end{equation}

Thus,
\begin{equation}\label{f12}
{\cal E}_{n}(C^{\psi}_{\beta,1})_{C}=
\frac{1}{\pi}
\sum\limits_{k=n}^{\infty}\psi(k)+
 \frac{\mathcal{O}(1)}{n} \sum\limits_{k=1}^{\infty}k\psi(k+n).
\end{equation}
Theorem~\ref{theorem1} is proved. 
 \end{proof}

\begin{corollary}\label{cor1}
Let the sequence $\psi(k)$, which generates the classes $C^{\psi}_{\beta,1}$, satisfies  the condition $D_{0}$, i.e.,   $\psi(k)>0$ and
$$
\lim\limits_{k\rightarrow 0}\frac{\psi(k+1)}{\psi(k)}=0.
$$
 Then, the following asymptotic equality holds as $n\rightarrow\infty$
\begin{equation}\label{f12}
{\cal E}_{n}(C^{\psi}_{\beta,1})_{C}=
\frac{1}{\pi}\psi(n)
+
\frac{\mathcal{O}(1)}{n}\sum\limits_{k=n+1}^{\infty}k\psi(k),
\end{equation}
where  $O(1)$ is a quantity uniformly bounded in all parameters.
\end{corollary}

\begin{proof}
Indeed, let $\psi\in D_{0}$, then the right hand of (\ref{Theorem1Asymp}) can be written in the form
\begin{align*}
{\cal E}_{n}(C^{\psi}_{\beta,1})_{C}&=
\frac{1}{\pi}\psi(n)
+
\mathcal{O}(1)\left( \sum\limits_{k=n+1}^{\infty}\psi(k)
+
 \frac{1}{n} \sum\limits_{k=0}^{\infty}k\psi(k+n)
\right)
\notag \\
&=
\frac{1}{\pi}\psi(n)
+
\frac{\mathcal{O}(1)}{n} \sum\limits_{k=1}^{\infty}(k+n)\psi(k+n)
\notag \\
&=
\frac{1}{\pi}\psi(n)
+
\frac{\mathcal{O}(1)}{n} \sum\limits_{k=n+1}^{\infty}k\psi(k).
\end{align*}
Corollary~\ref{cor1} is proved.
\end{proof}

Asymptotic equality (\ref{f12}) with  written in another form  reminder was obtained earlier in \cite{Serdyuk2005} and \cite{Serdyuk2005Lp}.

\begin{corollary}\label{cor01}
Let $\psi(k)=e^{-\alpha k^{-r}}$, $r>1$, $\alpha>0$ and $\beta\in\mathbb{R}$. Then as $n\rightarrow\infty$ the following asymptotic equality holds
\begin{equation}\label{f41}
{\cal E}_{n}(C^{\alpha,r}_{\beta,1})_{C}= 
e^{-\alpha n^{r}}\Big(
\frac{1}{\pi}+\mathcal{O}(1)e^{-\alpha r n^{r-1}}\Big(1+\frac{1}{\alpha r (n+1)^{r-1}}\Big)\Big),
\end{equation}
where  $O(1)$ is a quantity uniformly bounded in $n$ and $\beta$.
\end{corollary}
\begin{proof}
 Formula (\ref{f12}) implies that as $n\rightarrow\infty$
\begin{equation}\label{f42}
{\cal E}_{n}(C^{\alpha,r}_{\beta,1})_{C}= 
\frac{1}{\pi}e^{-\alpha n^{r}}
+
\frac{\mathcal{O}(1)}{n}\sum\limits_{k=n+1}^{\infty}ke^{-\alpha k^{r}}.
\end{equation}

It is easy to see that
\begin{equation}\label{f43}
\frac{1}{n}\sum\limits_{k=n+1}^{\infty}ke^{-\alpha k^{r}}
<\frac{1}{n}\left((n+1)e^{-\alpha (n+1)^{r}}+
\int\limits_{n+1}^{\infty}te^{-\alpha t^{r}}dt \right).
\end{equation}

Integrating by parts, we get
\begin{align}\label{f44}
\int\limits_{n+1}^{\infty}te^{-\alpha t^{r}}dt &=
\int\limits_{n+1}^{\infty}t^{2}  \frac{1}{\alpha r t^{r}} \left(-e^{-\alpha t^{r}}\right)' dt
\leq
 \frac{1}{\alpha r (n+1)^{r}} \int\limits_{n+1}^{\infty}t^{2} \left(-e^{-\alpha t^{r}}\right)' dt
 \notag \\
 &= \frac{1}{\alpha r (n+1)^{r}}  \left((n+1)^{2}e^{-\alpha (n+1)^{r}}+2\int\limits_{n+1}^{\infty}te^{-\alpha t^{r}}dt \right).
\end{align}
From the last equality we have
\begin{equation}\label{f45}
\left(1- \frac{2}{\alpha r (n+1)^{r}} \right)  \int\limits_{n+1}^{\infty}te^{-\alpha t^{r}}dt 
\leq
\frac{(n+1)^{2}e^{-\alpha (n+1)^{r}}}{\alpha r (n+1)^{r}},
\end{equation}
which is equivalent to
\begin{equation}\label{f46}
  \int\limits_{n+1}^{\infty}te^{-\alpha t^{r}}dt 
\leq
\frac{e^{-\alpha (n+1)^{r}}}{\alpha r (n+1)^{r-2}}\frac{\alpha r(n+1)^{r}}{\alpha r(n+1)^{r}-2}=
\frac{e^{-\alpha (n+1)^{r}}}{\alpha r (n+1)^{r-2}}\left(1+ \frac{2}{\alpha r(n+1)^{r}-2}\right).
\end{equation}
Relations (\ref{f43}) and (\ref{f46}) yield that
\begin{equation}\label{f47}
\frac{1}{n}\sum\limits_{k=n+1}^{\infty}ke^{-\alpha k^{r}}
=
\mathcal{O}\left( e^{-\alpha (n+1)^{r}}+
\frac{e^{-\alpha (n+1)^{r}}}{\alpha r (n+1)^{r-2}}\left(1+ \frac{2}{\alpha r(n+1)^{r}-2}\right)
\right).
\end{equation}

Combining (\ref{f42}) and (\ref{f47}) we obtain
\begin{align*}
{\cal E}_{n}(C^{\alpha,r}_{\beta,1})_{C}=& 
\frac{1}{\pi}e^{-\alpha n^{r}}
+
\mathcal{O}\left( e^{-\alpha (n+1)^{r}}+
\frac{e^{-\alpha (n+1)^{r}}}{\alpha r (n+1)^{r-2}}\left(1+ \frac{2}{\alpha r(n+1)^{r}-2}\right)
\right) \notag \\
=& 
e^{-\alpha n^{r}}\left(\frac{1}{\pi}
+
\mathcal{O}\left( e^{-\alpha r n^{r-1}}+
\frac{e^{-\alpha r n^{r-1}}}{\alpha r (n+1)^{r-2}}
\right)\right) .
\end{align*}
Corollary~\ref{cor01} is proved.
\end{proof}

Formula (\ref{f41}) was obtained in \cite{Serdyuk2005} and \cite{Serdyuk2005Lp}.
%The proof of Corollary~\ref{cor1} follows from Theorem~\ref{theorem1} and from proof of Theorem 4 \cite{Serdyuk2005}.

%The asymptotic equality (\ref{f12}) was proved in \cite{Serdyuk2005}.

%For every fixed $q\in(0,1)$ we denote by $D_{q}$ the set of all sequences $\psi(k)$, $k\in\mathbb{N}$, such that
%$$
%\lim\limits_{k\rightarrow 0}\frac{\psi(k)}{\psi(k+1)}=q.
%$$
For classes $C^{\alpha,1}_{\beta,1}$, generated by classes of Poisson kernels  
\begin{equation}\label{kernelPsiDq}
P_{\alpha,1,\beta}(t)=\sum\limits_{k=1}^{\infty}e^{-\alpha k}\cos
\big(kt-\frac{\beta\pi}{2}\big), \ \alpha>0, \  \beta\in
    \mathbb{R},
\end{equation}
the following statement holds.

%The classes $C^{\psi}_{\beta,1}$, generated by kernels (\ref{kernelPsiDq}) we denote by  $C^{q}_{\beta,1}$.

\begin{corollary}\label{cor2}
Let $\alpha>0$ and $\beta\in \mathbb{R}$.
 Then the following asymptotic equality holds as $n\rightarrow\infty$
\begin{equation}\label{f31}
{\cal E}_{n}(C^{\alpha,1}_{\beta,1})_{C}=e^{-\alpha n}\left(
\frac{1}{\pi}\frac{1}{1-e^{-\alpha}}
+
\frac{\mathcal{O}(1)}{n}\frac{e^{-\alpha}}{(1-e^{-\alpha})^{2}}\right),
\end{equation}
where  $O(1)$ is a quantity uniformly bounded in all parameters.
\end{corollary}

\begin{proof}
Denote $q=e^{-\alpha}$. Then, 
from Theorem~\ref{theorem1} it follows
\begin{align}\label{f32}
{\cal E}_{n}(C^{\alpha,1}_{\beta,1})_{C}=&
\frac{1}{\pi}
\sum\limits_{k=n}^{\infty}q^{k}+
 \frac{\mathcal{O}(1)}{n} \sum\limits_{k=0}^{\infty}kq^{k+n} \notag \\
 =&\frac{1}{\pi}\frac{q^{n}}{1-q}
+
\frac{\mathcal{O}(1)}{n}\left( \sum\limits_{k=n}^{\infty}kq^{k}-n\sum\limits_{k=n}^{\infty}q^{k}  \right) \notag \\
=&
\frac{1}{\pi}\frac{q^{n}}{1-q}
+
\frac{\mathcal{O}(1)}{n}\left( \frac{nq^{n}(1-q)+q^{n+1}}{(1-q)^{2}}-\frac{nq^{n}}{1-q} \right) \notag \\
=&
\frac{1}{\pi}\frac{q^{n}}{1-q}
+
\frac{\mathcal{O}(1)}{n}\frac{q^{n+1}}{(1-q)^{2}},
 \end{align}
 where we have used that
 \begin{equation*}
 \sum\limits_{k=n}^{\infty}kq^{k}=\frac{nq^{n}(1-q)+q^{n+1}}{(1-q)^{2}}.
 \end{equation*}
 Corollary~\ref{cor2} is proved.
\end{proof}

The asymptotic equality (\ref{f31}) was proved in \cite{Serdyuk2005} and \cite{Serdyuk2005Lp}.

\begin{corollary}\label{cor00}
Let $\psi(k)=e^{-\alpha k^{-r}}$, $0<r<1$, $\alpha>0$, $\beta\in\mathbb{R}$. Then as $n\rightarrow\infty$ the following asymptotic equality holds
\begin{equation}\label{f112}
{\cal E}_{n}(C^{\alpha,r}_{\beta,1})_{C}= 
\frac{e^{-\alpha n^{r}}}{\pi\alpha r}n^{1-r}\Big(
1+\mathcal{O}(1)\Big(\frac{1}{n^{r}}+\frac{1}{n^{1-r}}\Big)\Big),
\end{equation}
where  $O(1)$ is a quantity uniformly bounded in $n$ and $\beta$.
\end{corollary}
\begin{proof}
 Theorem~\ref{theorem1} allows to write
\begin{equation}\label{f33}
{\cal E}_{n}(C^{\psi}_{\beta,1})_{C}=
\frac{1}{\pi}
\sum\limits_{k=n}^{\infty}e^{-\alpha k^{r}}+
 \frac{\mathcal{O}(1)}{n} \sum\limits_{k=0}^{\infty}ke^{-\alpha (k+n)^{r}}.
 \end{equation}

Formulas (90) and (91) of the work \cite{SerdyukStepanyuk2018} imply that
\begin{equation}\label{z1}
 \sum\limits_{k=0}^{\infty}e^{-\alpha (k+n)^{r}}=
\frac{e^{-\alpha n^{r}}}{\alpha r}n^{1-r}\Big(
1+\mathcal{O}\Big(\frac{1}{n^{r}}+\frac{1}{n^{1-r}}\Big)\Big). \ \
 \end{equation}
 From formulas  (94) and (97) of the work \cite{SerdyukStepanyuk2018} it follows that
 \begin{equation}\label{p1}
\frac{1}{n}\sum\limits_{k=1}^{\infty}k e^{-\alpha(k+n)^{r}}=\mathcal{O}(1)
\frac{1}{n}e^{-\alpha n^{r}}(n^{2-2r}+n)=
\mathcal{O}(1)
e^{-\alpha n^{r}}n^{1-r}\Big(\frac{1}{n^{r}}+\frac{1}{n^{1-r}}\Big).
\end{equation}
Combining (\ref{f33})--(\ref{p1}) we obtain (\ref{f12}).
Corollary~\ref{cor00} is proved. 
 \end{proof}
Asymptotic equality (\ref{f12}) was proved in \cite{SerdyukStepanyuk2018}.

By $\mathfrak{M}$ we denote the set of all convex (downward) continuous functions $\psi(t)$, $t\geq 1$, such that $\lim\limits_{t\rightarrow\infty}\psi(t)=0$.

Assume that the sequence $\psi(k)$, $k\in\mathbb{N}$,  specifying the class $C^{\psi}_{\beta,1}$ is the restriction of the functions $\psi(t)$ from $\mathfrak{M}$ to the set of natural numbers.

We also consider the following characteristics of functions $\psi\in\mathfrak{M}$:
 \begin{equation*}
\alpha(t):=\frac{\psi(t)}{t|\psi'(t)|}
\end{equation*}
and
\begin{equation*}
\lambda(t):=\frac{\psi(t)}{|\psi'(t)|}.
\end{equation*}

\begin{theorem}\label{theorem2}
Let $\psi\in\mathfrak{M}$, $\alpha(t)\downarrow0$, $\lambda(t)\rightarrow \infty$, $\lambda'(t)\rightarrow 0$ as $t\rightarrow\infty$ and $\beta\in\mathbb{R}$. Then as $n\rightarrow\infty$ the following asymptotic equality holds
\begin{equation}\label{Theorem2Asymp}
{\cal E}_{n}(C^{\psi}_{\beta,1})_{C}=
\psi(n) \lambda(n)
\left( \frac{1}{\pi}
+
 \mathcal{O}\left(\frac{1}{\lambda(n)}+\alpha(n)+\varepsilon_{n}  \right)
 \right),
 \end{equation}
 where  $\varepsilon_{n}:=\sup\limits_{t\geq n}|\lambda'(t)|$ and
 $O(1)$ is a quantity uniformly bounded in $n$ and $\beta$.
\end{theorem}

\begin{proof}
From Theorem~\ref{theorem1} we have that
 the following asymptotic equality holds as $n\rightarrow\infty$
\begin{equation}\label{formTheor1}
{\cal E}_{n}(C^{\psi}_{\beta,1})_{C}=
\frac{1}{\pi}\sum\limits_{k=n}^{\infty}\psi(k)
+
\frac{\mathcal{O}(1)}{n}\sum\limits_{k=1}^{\infty}k\psi(k+n).
\end{equation}
Notice that
\begin{equation}\label{form6}
\sum\limits_{k=0}^{\infty}k\psi(k+n)=\sum\limits_{k=n}^{\infty}k\psi(k)
-n\sum\limits_{k=n}^{\infty}\psi(k)<
\psi(n)n+\int\limits_{n}^{\infty}t\psi(t) dt-n\int\limits_{n}^{\infty}\psi(t) dt.
 \end{equation}

Let $\lambda'(t)\rightarrow 0$ as $t\rightarrow\infty$. Then integrating by parts, we get
\begin{align*}
I_{1}:=& \int\limits_{n}^{\infty}\psi(t) dt=
 \int\limits_{n}^{\infty}-\psi'(t)\lambda(t) dt=
 \psi(n)\lambda(n)+\int\limits_{n}^{\infty}\psi(t)\lambda'(t) dt \notag \\
= &
  \psi(n)\lambda(n)+\lambda'(\theta_{1}) I_{1},
\end{align*}
where $\theta_{1}$ is a some point from the interval $[n,\infty)$.

Then
\begin{equation*}
I_{1}\left( 1- \lambda'(\theta_{1})\right)=  \psi(n)\lambda(n)
 \end{equation*}
 and
 \begin{equation*}
I_{1}= \psi(n)\lambda(n)\left( 1+ \frac{\lambda'(\theta)}{1-\lambda'(\theta)}\right)=
\psi(n)\lambda(n)\left( 1+\mathcal{O}(1)\varepsilon_{n}\right).
 \end{equation*}
 Again integrating by parts
 \begin{align*}
I_{2}:=& \int\limits_{n}^{\infty}t\psi(t) dt=
  \int\limits_{n}^{\infty}t^{2}  \frac{\psi(t)}{-t\psi'(t)} (-\psi'(t))dt=
 \alpha(\theta_{2}) \int\limits_{n}^{\infty}t^{2}(-\psi'(t))dt
 \notag \\
=&\alpha(\theta_{2}) \left(n^{2}\psi(n)+2\int\limits_{n}^{\infty}t\psi(t) \right),
\end{align*}
where $\theta_{2}$ is a some point from the interval $[n,\infty)$.

Assume that $\alpha(t)$ monotonically decreases. Then
 \begin{equation*}
I_{2}\leq \alpha(n) n^{2}\psi(n)+2\alpha(n)I_{2},
 \end{equation*}
 which is equivalent to
  \begin{equation*}
I_{2}\left(1-2\alpha(n) \right)\leq \alpha(n)n^{2}\psi(n)=\psi(n)n\frac{\psi(n)}{|\psi'(n)|}
 \end{equation*}
 and
  \begin{equation*}
\frac{1}{n}I_{2}\leq \psi(n)\frac{\psi(n)}{|\psi'(n)|} \frac{1}{1-2\alpha(n)}.
 \end{equation*}
 Hence, if $\alpha(t)\downarrow 0$, then
 \begin{equation}\label{int1}
\frac{1}{n}I_{2}\leq 
\psi(n)\frac{\psi(n)}{|\psi'(n)|} \left(1+\frac{2\alpha(n)}{1-2\alpha(n)}\right)
 \end{equation}
 and
  \begin{equation}\label{int2}
I_{1}= 
\psi(n)\frac{\psi(n)}{|\psi'(n)|} +
\psi(n)\frac{\psi(n)}{|\psi'(n)|} \mathcal{O}(\varepsilon_{n}).
 \end{equation}
 Combining (\ref{int1}) and (\ref{int2}), we obtain
 \begin{align}\label{addit1}
 \frac{1}{n}\sum\limits_{k=0}^{\infty}k\psi(k+n)\leq& \psi(n)+
\psi(n)\frac{\psi(n)}{|\psi'(n)|} \left(1+\frac{2\alpha(n)}{1-2\alpha(n)}\right) \notag \\
-&
\left( \psi(n)\frac{\psi(n)}{|\psi'(n)|} +
\psi(n)\frac{\psi(n)}{|\psi'(n)|} \mathcal{O}(\varepsilon_{n}) \right) \notag \\
 =&
 \psi(n)+
\psi(n)\frac{\psi(n)}{|\psi'(n)|} \left(\frac{2\alpha(n)}{1-2\alpha(n)}+ \mathcal{O}(\varepsilon_{n})
\right) \notag \\
=&
\psi(n)\frac{\psi(n)}{|\psi'(n)|}
  \mathcal{O} \left(\frac{1}{\lambda(n)}+\alpha(n)+\varepsilon_{n}
\right) .
 \end{align}
 Moreover,
  taking into account (\ref{int2}),
  \begin{align}\label{addit}
   \sum\limits_{k=n}^{\infty}\psi(k)=I_{1}+\mathcal{O}(1)\psi(n)=
  \psi(n)\frac{\psi(n)}{|\psi'(n)|}
  \left(1+ 
\mathcal{O}(1)\left(\varepsilon_{n} +\frac{1}{\lambda(n)}\right)\right).
\end{align}
Formulas (\ref{formTheor1}), (\ref{addit1}) and (\ref{addit}) imply (\ref{Theorem2Asymp}).
 Theorem~\ref{theorem2} is proved.
\end{proof}

\begin{corollary}\label{cor3}
Let $\psi(k)=(k+2)^{- \ln\ln (k+2)}$, $\beta\in\mathbb{R}$ and $k\in\mathbb{N}$. Then as $n\rightarrow\infty$ the following asymptotic equality holds
\begin{equation}\label{f13}
{\cal E}_{n}(C^{\psi}_{\beta,1})_{C}=
\frac{1}{\pi}\psi(n)\frac{n}{\ln\ln (n+2)}+\mathcal{O}(\psi(n)).
\end{equation}
\end{corollary}
\begin{proof}
If $\psi(k)=e^{-\ln (k+2) \ln\ln (k+2)}$, then
\begin{align*}
\psi'(t)&=
-e^{-\ln (t+2) \ln\ln (t+2)}\left( \frac{\ln\ln (t+2)}{t+2}+\frac{\ln (t+2)}{(t+2)\ln (t+2)} \right)
\notag \\
&=
-e^{-\ln (t+2) (\ln\ln (t+2)}\frac{\ln\ln (t+2)+1}{t+2}.
\end{align*}
Doing elementary calculations, we get
\begin{equation}\label{f14}
\lambda(t)=\frac{t+2}{\ln\ln (t+2)+1}=
\frac{t}{\ln\ln (t+2)}+\mathcal{O}(1),
\end{equation}
\begin{equation}\label{f-14}
\alpha(t)=\frac{t+2}{t}\frac{1}{\ln\ln( t+2)+1}
\end{equation}
and
\begin{equation}\label{f15}
\lambda'(t)=\frac{\ln\ln (t+2)+1- \frac{1}{\ln (t+2)}}{(\ln\ln (t+2)+1)^{2}}\leq 
\frac{1}{\ln\ln (t+2)}.
\end{equation}
Substituting (\ref{f14})--(\ref{f15}) in (\ref{Theorem2Asymp}) we obtain (\ref{f13}).
Corollary~\ref{cor3} is proved.
\end{proof}

\begin{corollary}\label{cor4}
Let $\psi(k)=e^{-\ln ^{2}k }$, $\beta\in\mathbb{R}$ and $k\in\mathbb{N}$. Then as $n\rightarrow\infty$ the following asymptotic equality holds
\begin{equation}\label{f16}
{\cal E}_{n}(C^{\psi}_{\beta,1})_{C}=
\frac{1}{2\pi}\frac{\psi(n) n}{\ln n}
+\mathcal{O}(\psi(n)).
\end{equation}
\end{corollary}
\begin{proof}
It is easy to see 
\begin{equation}\label{f17}
\psi'(t)=
-2\frac{1}{t} e^{-\ln ^{2}t }\ln t .
\end{equation}
Formula (\ref{f17}) yields
\begin{equation}\label{f18}
\lambda(t)=\frac{t}{2\ln t}, \ \ \alpha(t)=\frac{1}{2\ln t}
\end{equation}
and
\begin{equation}\label{f19}
\lambda'(t)=\frac{\ln t-1}{2(\ln t)^{2}}\leq \frac{1}{2\ln t}.
\end{equation}
Formulas (\ref{f17})--(\ref{f19}) and (\ref{Theorem2Asymp}) imply (\ref{f16}).
Corollary~\ref{cor4} is proved.
\end{proof}

\begin{corollary}\label{cor5}
Let $\psi(k)=e^{- \frac{k+1}{\ln (k+1)}}$, $\beta\in\mathbb{R}$ and $k\in\mathbb{N}$. Then as $n\rightarrow\infty$ the following asymptotic equality holds
\begin{equation}\label{f20}
{\cal E}_{n}(C^{\psi}_{\beta,1})_{C}=
\psi(n)\ln (n+1) \left(\frac{1}{\pi}+\mathcal{O}\left(\frac{1}{\ln(n+1)} \right) \right).
\end{equation}
\end{corollary}

\begin{proof}
Doing elementary calculations, we get
\begin{equation*}
\psi'(t)=-e^{- \frac{t+1}{\ln (t+1)}}
\frac{\ln (t+1)-1}{\ln^{2}(t+1)},
\end{equation*}
\begin{equation}\label{f21}
\lambda(t)=\frac{\ln^{2} (t+1)}{\ln (t+1)-1}, \  \
 \alpha(t)=\frac{\ln^{2}(t+1)}{t\ln (t+1) -t}=\mathcal{O}\left(\frac{1}{t\ln (t+1)} \right)
\end{equation}
and
\begin{equation}\label{f22}
\lambda'(t)=\frac{1}{t+1}-\frac{1}{(t+1)(\ln (t+1)-1)^{2}}\leq \frac{1}{t+1}.
\end{equation}
Formulas (\ref{f21}), (\ref{f22}) and (\ref{Theorem2Asymp}) imply (\ref{f20}).
Corollary~\ref{cor5} is proved.
\end{proof}

\section*{Acknowledgements}

Second author is supported by the Austrian Science Fund FWF project  F5506-N26 (part of the Special Research Program (SFB) ``Quasi-Monte Carlo Methods: Theory and Applications'') and partially is supported by grant of NASU for  groups of young scientists (project No16-10/2018)


\begin{thebibliography}{15}






%\bibitem {Serdyuk2005} A.S. Serdyuk,  Approximation of classes of analytic functions by Fourier sums in the uniform metric, Ukr. Math. J. 57:8  (2005) 1079-1096.

%\bibitem{SerdyukStepanyuk2016}
%A. S. Serdyuk and T. A. Stepanyuk, "Uniform approximations by Fourier sums on the classes of convolutions with generalized Poisson kernels," Dop. Nats. Akad. Nauk. Ukr., No. 11, 10–16 (2016).

%\bibitem{SerdyukStepanyuk2017}
%A.S. Serdyuk, T.A. Stepanyuk, Approximations by Fourier sums of classes of generalized Poisson integrals in metrics of spaces $L_{s}$,  Ukr. Mat. J. 69:5, 811-822 (2017).


\bibitem{Kol}
 A. Kolmogoroff, Zur Gr\"{o}ssennordnung des Restgliedes
Fourierschen Reihen differenzierbarer Funktionen, (in German) Ann. Math.(2),
36:2   (1935) 521--526.


\bibitem{Korn}
N.P. Korneichuk, Exact Constants in Approximation Theory, Vol. 38, Cambridge Univ. Press, Cambridge, New York 1990.

%\cite{Kol}, \cite{Nikolsky 1946}--\cite{Serdyuk_Stepaniuk2015}, \cite{Stepanets1}

\bibitem{Nikolsky 1946}
S.M. Nikol'skii,
Approximation of functions in the mean by trigonometrical polynomials, (in Russian)
Izv. Akad. Nauk SSSR, Ser. Mat. 10  (1946) 207-256.

\bibitem{Serdyuk2005}
A.S. Serdyuk, Approximation of classes of analytic functions by Fourier sums in uniform metric,  Ukr. Math. J. 57:8 (2005) 1275--1296.

\bibitem{Serdyuk2005Lp}
A.S. Serdyuk,  Approximation of classes of analytic functions by Fourier sums in the metric of the space $L_p$,
 Ukr. Math. J. 57:10 (2005)1395--1408.




\bibitem{Serdyuk_Stepaniuk2015}
A.S. Serdyuk, T.A.Stepanyuk, Order estimates for the best approximations and approximations by Fourier sums of the classes
of convolutions of periodic functions of low smoothness in the integral metric, Ukr. Math. J. 65:12  (2015) 1862--1882.


\bibitem{SerdyukStepanyuk2018}
A.S. Serdyuk, T.A. Stepanyuk, Uniform approximations by Fourier sums on  classes of generalized Poisson integrals, Analysis Mathematica 45 (1) (2019), 201--236.

\bibitem{Stepanets1}
A.I. Stepanets,
Methods of Approximation Theory. VSP: Leiden, Boston  2005.

\bibitem{Stechkin 1980}
S.B. Stechkin,
An estimate of the remainder term of Fourier series for differentiable functions, (in Russian)
Tr. Mat. Inst. Steklova 145  (1980) 126--151. 

\bibitem{Teljakovsky1968}
S.A. Telyakovskii, Approximation of differentiable functions by partial sums of their Fourier series, Mathematical Notes 4:3  (1968) 668--673.


%\bibitem {Stepanets1989N4} A.I. Stepanets, 
%On the Lebesgue inequality on classes of  $(\psi,\beta)$-differentiable functions, Ukr. Math. J. 41:4 435--443 (1989).

%\bibitem {StepanetsSerdyuk} A.I. Stepanets, A.S. Serdyuk, Lebesgue inequalities for Poisson integrals, Ukr. Math. J. 52:6, 798-808 (2000).

%\bibitem {SerdyukMusienko} A.S. Serdyuk, A.P. Musienko, The Lebesgue type inequalities for the de la Vall?e Poussin sums in approximation of Poisson integrals,
%Zb. Pr. Inst. Mat. NAN Ukr. 7:1, 298-316 (2010).




\end{thebibliography}
\end{document}